\documentclass[reqno,10pt]{amsart}

\usepackage[english]{babel}
\usepackage{lipsum}
\usepackage{fancyhdr}
\usepackage{csquotes}
\usepackage{graphicx}
\usepackage{listings}
\usepackage{xcolor}
\usepackage{amsthm}
\usepackage{multicol}
\usepackage{amsmath}
\usepackage{amssymb}
\usepackage{accents}
\usepackage{stackengine}
\usepackage{a4wide}
\usepackage{palatino}
\usepackage{mathtools}
\mathtoolsset{showonlyrefs,showmanualtags}
\definecolor{citegreen}{rgb}{0,0.6,0}
\definecolor{refred}{rgb}{0.8,0,0}
\usepackage[colorlinks, citecolor=citegreen,linkcolor=refred]{hyperref}
\usepackage{mathrsfs}
\usepackage{tikz}
\usepackage{pgfplots}
\usepackage{adjustbox}
\usepackage{braket}
\usepackage{esint}
\pgfplotsset{/pgf/number format/use comma,compat=newest}
\usetikzlibrary{patterns}
\graphicspath{ {./images/} }

\theoremstyle{plain}
\newtheorem{Teorema}{Theorem}[section]

\newtheorem{ackn}{Acknowledgements\hspace{-.2em}}

\theoremstyle{definition}
\newtheorem{definizione}[Teorema]{Definition}

\theoremstyle{remark}
\newtheorem{oss}[Teorema]{Remark}

\numberwithin{equation}{section}

\definecolor{codegreen}{rgb}{0,0.6,0}
\definecolor{codegray}{rgb}{0.5,0.5,0.5}
\definecolor{codepurple}{rgb}{0.58,0,0.82}
\definecolor{backcolour}{rgb}{0.95,0.95,0.92}

\lstdefinestyle{mystyle}{
 backgroundcolor=\color{backcolour}, 
 commentstyle=\color{codegreen},
 keywordstyle=\color{magenta},
 numberstyle=\tiny\color{codegray},
 stringstyle=\color{codepurple},
 basicstyle=\ttfamily\footnotesize,
 breakatwhitespace=false, 
 breaklines=true, 
 captionpos=b, 
 keepspaces=true, 
 numbers=left, 
 numbersep=5pt, 
 showspaces=false, 
 showstringspaces=false,
 showtabs=false, 
 tabsize=2
}

\def\R{{{\mathbb R}}}

\def\NN{{{\mathbb N}}}

\def\L{{{\mathscr L}}}

\def\XXint#1#2#3{{\setbox0=\hbox{$#1{#2#3}{\int}$} 
\vcenter{\hbox{$#2#3$}}\kern-.5\wd0}}

\newcommand{\weakkk}{\mathrel{\ensurestackMath{\stackon[-.25em]{\xrightharpoonup{\ \ \,}}{\scriptstyle{\,\star\,}}}}} 

\definecolor{MyDarkGreen}{rgb}{0,0.50,0.04}

\title{Quasiconvexity in the Riemannian setting}

\date{\today}
 
\author[Aurora Corbisiero]{Aurora Corbisiero}
\address[Aurora Corbisiero]{Department of Mathematics and Applications ``Renato Caccioppoli'', Universit\`a di Napoli Federico II}
\email{au.corbisiero@studenti.unina.it}

\author[Chiara Leone]{Chiara Leone}
\address[Chiara Leone]{Department of Mathematics and Applications ``Renato Caccioppoli'', Universit\`a di Napoli Federico II}
\email{chileone@unina.it}

\author[Carlo Mantegazza]{Carlo Mantegazza}
\address[Carlo Mantegazza]{Department of Mathematics and Applications ``Renato Caccioppoli'', Universit\`a di Napoli Federico II \& Scuola Superiore Meridionale, Naples, Italy}
\email[C. Mantegazza]{carlo.mantegazza@unina.it}

\makeatletter
\patchcmd{\@setaddresses}{\indent}{\noindent}{}{}
\patchcmd{\@setaddresses}{\indent}{\noindent}{}{}
\patchcmd{\@setaddresses}{\indent}{\noindent}{}{}
\makeatother

\begin{document}

\begin{abstract}
We introduce a notion of quasiconvexity for continuous functions $f$ defined on the vector bundle of linear maps between the tangent spaces of a smooth Riemannian manifold $(M,g)$ and $\mathbb{R}^m$, naturally generalizing the classical Euclidean definition. We prove that this condition characterizes the sequential lower semicontinuity of the associated integral functional
\[
F(u, \Omega) = \int_{\Omega} f(du) \, d\mu
\]
with respect to the weak$^*$ topology of $W^{1,\infty}(\Omega, \mathbb{R}^m)$, for every bounded open subset $\Omega\subseteq M$.
\end{abstract}

\maketitle
\tableofcontents

\section{Introduction}
A fundamental result in the calculus of variations concerning the search for minimizers of variational functionals is the characterization of the sequential lower semicontinuity of integral functionals defined on Sobolev spaces $W^{1,p}(\Omega, \mathbb{R}^m)$, where $\Omega \subseteq \mathbb{R}^n$ is open and bounded. Specifically, for $1 \leqslant p \leqslant +\infty$, lower semicontinuity with respect to the weak topology (or weak$^*$ topology if $p=+\infty$) is equivalent to the quasiconvexity of the integrand, provided certain growth conditions are satisfied.

Historically, this equivalence was established through several steps: in the pioneering paper~\cite{articoloTonelli} by Tonelli, it is shown that for a twice differentiable and continuous function 
$$
f: [a,b]\times\mathbb{R}^m\times\mathbb{R}^m\to [0,+\infty)
$$
and for $u\in W^{1,1}((a,b),\mathbb{R}^m)$, the functional 
\begin{equation*}
F(u)=\int_{a}^bf(x, u(x), u'(x))\, dx
\end{equation*}
is sequentially weakly lower semicontinuous in $W^{1,1}(a,b)$ if and only if the function $f$ is convex in the third variable. In the scalar case $m=1$, this result was later generalized to functions defined on bounded open sets of $\mathbb{R}^n$ by various authors and Serrin in~\cite{Serrin15} proved that the differentiability assumptions are not actually required. Subsequent improvements of Serrin's theorem were given by De~Giorgi~\cite{degiorgi}, Olech~\cite{olech} and Ioffe~\cite{ioffe}. 

The results in the scalar case extend easily to the vectorial case. However, while for $m=1$ the semicontinuity theorem stated above is optimal, in the sense that the convexity assumption of $f(x,s,\xi)$ with respect to $\xi$ is necessary for the lower semicontinuity of $F$, in the vectorial case for 
$m>1$, there are functionals (of considerable interest in the theory of nonlinear elasticity, for instance) that are lower semicontinuous without $f$ being convex with respect to the matrix $\xi=\xi_{ij}$. 

In the case of $m>1$, the condition on $f$ that turned out to be necessary and, with additional assumptions, also sufficient for the lower semicontinuity of integral functionals, is quasiconvexity, introduced by Morrey~\cite{Morrey12} in~1952. 

\begin{definizione}\label{quasidef}
A continuous function $f:\mathbb{R}^{{n\times m}}\to \mathbb{R}$ is {\em quasiconvex} if for every ${\xi}\in \mathbb{R}^{{n\times m}} $ and for every open subset $\Omega$ of $\mathbb{R}^{n}$, there holds
\begin{equation}\label{quasiconvexdef}
f({\xi})\leqslant \fint_{\Omega}f({\xi}+D\varphi(x))\,d\L^n(x),
\end{equation}
for every function $\varphi\in C^{\infty}_c(\Omega, \mathbb{R}^m)$.\\
A real function $f: \mathbb{R}^n\times \mathbb{R}^m \times \mathbb{R}^{{n\times m}}\to \mathbb{R}$ is {\em quasiconvex} in the variable $\xi\in \mathbb{R}^{n\times m}$ if there exists a subset $Z$ of $\mathbb{R}^n$ with $\L^n(Z)=0$, such that for every ${x}\in \mathbb{R}^n\setminus Z$ and for every ${s}\in \mathbb{R}^m$ the function $\xi \mapsto f({x},{s},\xi )$ is {\em quasiconvex.}
\end{definizione}

More precisely, Morrey showed that under some strong regularity assumptions on the function $f$, the equivalence between its quasiconvexity and the sequential weak$^*$ lower semicontinuity in $ W^{1,\infty}(\Omega, \mathbb{R}^m)$ of the functional 
\begin{equation}\label{Fmorrey}
u\mapsto F(u,\Omega)=\int_{\Omega}f(x,u(x), D u(x))\,d\L^n(x)
\end{equation}
holds. Meyers then extended Morrey’s result to the setting of $W^{k,p}(\Omega, \mathbb{R}^m)$ spaces in~\cite{Meyers}.

Acerbi and Fusco in~\cite{AcerbiFusco} obtained a significant improvement of this result: they indeed established such equivalence for Carath\'eodory integrands with appropriate growth conditions in $W^{1,p}(\Omega, \mathbb{R}^m)$, for $1\leqslant p\leqslant +\infty$. We also mention that Marcellini in~\cite{marcellini} presented an alternative proof of this fact. We list below the theorems proved by Acerbi and Fusco.

\begin{Teorema}[Acerbi--Fusco]\label{Fusco1}
Let $f: \mathbb{R}^n\times\mathbb{R}^m\times\mathbb{R}^{{n\times m}}\to \mathbb{R}$ be a Carath\'eodory function satisfying
$$
0\leqslant f(x, s, \xi)\leqslant a(x)+b(s, \xi),
$$
for almost every $x\in \mathbb{R}^n$, $s\in \mathbb{R}^m$ and $\xi\in \mathbb{R}^{{n\times m}}$, where $a:\R^n\to\R$ is nonnegative and locally summable and $b:\R^n\times\mathbb{R}^{{n\times m}}\to\R$ is nonnegative and locally bounded.\\
Then, $f$ is quasiconvex in $\xi$ if and only if for every open bounded set $\Omega$ in $\mathbb{R}^n$ the functional $u\mapsto F(u, \Omega)$ is sequentially weakly$^*$ lower semicontinuous on $W^{1,\infty}(\Omega, \mathbb{R}^m)$.
\end{Teorema}

\begin{Teorema}[Acerbi--Fusco]\label{Teocap1pfinito}
 Let $1\leqslant p<+\infty$ and $f: \mathbb{R}^n\times\mathbb{R}^m\times\mathbb{R}^{{n\times m}}\to \mathbb{R}$ be a Carath\'eodory function satisfying
 \begin{equation}\label{II.6}
 0\leqslant f(x,s,\xi)\leqslant a(x)+C\big(|s|^p+|\xi|^p\big),
 \end{equation}
 for almost every $x\in \mathbb{R}^n$, $s\in \mathbb{R}^m$ and $\xi\in \mathbb{R}^{{n\times m}}$, where $a:\R^n\to\R$ is nonnegative and locally summable and $C$ is a nonnegative constant.\\
Then, $f$ is quasiconvex in $\xi$ if and only if for every open bounded set $\Omega$ in $\mathbb{R}^n$ the functional $u\mapsto F(u, \Omega)$ is sequentially weakly lower semicontinuous on $W^{1,p}(\Omega, \mathbb{R}^m)$. 
\end{Teorema}

Our aim is to develop an analogue of this theory in the Riemannian setting. Precisely, we will consider a smooth, complete and connected Riemannian manifold $(M,g)$ and a continuous function
\begin{equation*}
f:\mathscr{L}(TM, \mathbb{R}^m)\to \mathbb{R},
\end{equation*}
where $\mathscr{L}(TM, \mathbb{R}^m)$ is the vector bundle of the linear maps between a tangent space of $M$ and $\R^m$, namely
$$
\mathscr{L}(TM, \mathbb{R}^m)=\big\{\alpha:T_xM\to \mathbb{R}^m\,\,\big|\,\, \text{$x\in M$ and $\alpha$ is linear}\big\}.
$$
Then, after introducing a generalization of the notion of quasiconvexity (extending the usual one in the case of the Euclidean spaces), we will show that $f$ is quasiconvex in our sense if and only if for every open and bounded subset $\Omega\subseteq M$, the functional 
$$
u\mapsto F(u,\Omega)=\int_\Omega f(du)\,d\mu
$$
(where $\mu$ is the canonical measure of $(M,g)$) is sequentially lower semicontinuous in the weak* topology of $W^{1,\infty}(\Omega,\R^m)$, analogously to the Euclidean case.

At the end of the paper we discuss some open problems and possible future research directions in this Riemannian context.

\begin{ackn}
The authors wish to thank Luigi~Ambrosio, Andrea~Braides, and Nicola~Fusco for their valuable suggestions, insightful comments and for carefully reading the preliminary versions of this manuscript.\\
The last two authors are members of the INdAM--GNAMPA and partially supported by the PRIN Project 2022E9CF89 ``GEPSO -- Geometric Evolution Problems and Shape Optimization''.
\end{ackn}

\section{Quasiconvexity in the Riemannian setting}\label{quasisec}

In this section, after introducing a suitable ``Riemannian'' definition of quasiconvexity generalizing the ``classical'' Euclidean one, we will prove that the sequential weak$^*$ lower semicontinuity in $W^{1,\infty}$ of a functional holds if and only if the integrand is quasiconvex, according to such definition.

\medskip

Let $(M,g)$ be a smooth, connected and complete Riemannian manifold of dimension $n$. We fix some notation for this section. 

For every $x\in M$ and $r>0$ such that $\sqrt{n\,}r/2$ is smaller than the injectivity radius of $M$ at $x$, we define the cube
\begin{equation*}
\widetilde{Q}^r_x=r\bigg(-\frac{1}{2}, \frac{1}{2} \bigg)^n\subseteq T_{x}M\approx\R^n
\end{equation*}
(where the identification of $T_xM$ with $\R^n$ is done via an orthonormal basis of $T_xM$) and its diffeomorphic image
\begin{equation*}
 Q^r_{x}=\operatorname{exp}_{x}\big(\widetilde{Q}^r_x\big)\subseteq M
\end{equation*}
by means of the exponential map
$$
\operatorname{exp}_{x}: T_{x}M\to M.
$$

We now define $\mathscr{L}(TM, \mathbb{R}^m)$ as the space of the linear maps between a tangent space of $M$ and $\R^m$, namely
$$
\mathscr{L}(TM, \mathbb{R}^m)=\big\{\alpha:T_xM\to \mathbb{R}^m\,\,\big|\,\, \text{$x\in M$ and $\alpha$ is linear}\big\}.
$$
This space has a natural structure of vector bundle via the local parametrizations of $M$, for instance if $m=1$ it coincides with $T^*M$ and in general it is the union of all the spaces $(T_x^*M)^m$ for $x\in M$ (which are the fibers on the points of $M$). It is clearly locally diffeomorphic to $\R^n\times\R^{n\times m}$, indeed, choosing a local orthonormal frame $E_1,\dots,E_n$ in a neighbourhood $U\subseteq M$ (diffeomorphic to $\R^n$) of a point $x_0\in M$, in order to locally ``trivialize'' the vector bundle, we define the map $I:\mathscr{L}(TU, \mathbb{R}^m)\to\R^{n\times m}$ as
\begin{equation}\label{mapI}
\alpha\mapsto I\alpha=\big(\alpha(E_1),\dots,\alpha(E_n)\big)\in(\R^m)^n,
\end{equation}
for every $\alpha\in\mathscr{L}(TU, \mathbb{R}^m)$, hence the map sending any $\alpha:T_xM\to\R^m$ to 
$$
(x,I\alpha)=\big(x,\alpha(E_1),\dots,\alpha(E_n)\big)\in U\times(\R^m)^n\approx\R^n\times\R^{n\times m}
$$
is a diffeomorphism.\\
Then, considering the standard quadratic norm $\Vert\cdot\Vert$ on $(\R^m)^n\approx \R^{n\times m}$, we have a distance $\delta$ on $\mathscr{L}(TU, \mathbb{R}^m)$ defined as follows: 
\begin{equation}\label{distL}
\delta(\alpha,\beta)=d^M(x,y)+\Vert I\alpha-I\beta\Vert=d^M(x,y)+\sqrt{\sum_{i=1}^n\sum_{j=1}^m \big|\alpha^j(E_i)-\beta^j(E_i)\big|^2\,}
\end{equation}
for every $\alpha,\beta\in\mathscr{L}(TU, \mathbb{R}^m)$ such that $\alpha\in (T_x^*M)^m$ and $\beta\in (T_y^*M)^m$, where $d^M$ is the Riemannian distance on $M$. It is then easy to see that such distance, restricted to every fiber $(T_x^*M)^m$ coincides with the one associated to the ``quadratic'' norm (equivalent to the operator norm) induced by the metric tensor $g$ of $M$, that we will denote with $|\cdot|_{g_x}$.\\
Clearly, if we have a function $u: M\to\R^m$, its differential map $x\mapsto du[x]:T_xM\to\R^m$ is a section of $\mathscr{L}(TM, \mathbb{R}^m)$.

\medskip

{\em We chose to adopt the notation $du[x]$ for the differential of a map $u: M\to\R^m$ at a point $x\in M$, for the sake of clarity in the computations that follow.}

\begin{definizione}\label{RiemannianQC}
Let $(M,g)$ be a smooth, connected and complete Riemannian manifold and $\mu$ its canonical volume measure. A continuous function $f: \mathscr{L}(TM, \mathbb{R}^m)\to \mathbb{R}$ is called {\em quasiconvex} if for every $x_0\in M $, $\alpha_{x_0}\in (T_{x_0}M)^m\subseteq\mathscr{L}(TM, \mathbb{R}^m)$ and $\varphi\in C^\infty_c(Q_{x_0}^r,\mathbb{R}^m)$, there holds
\begin{equation}\label{quasidefR}
f(\alpha_{x_0})\leqslant\fint_{Q_{x_0}^r} f\big(\alpha_{x_0}+d\varphi [x]\circ d\exp_{x_0}[\exp_{x_0}^{-1}(x)]\big)\,d\mu(x)+o(1),
\end{equation}
where $o(1)$ is a function which goes to zero as $r\to 0$ and depends in a monotonically nondecreasing way only on the $L^{\infty}$ norm of $d\varphi$ (for fixed $x_0$ and $\alpha_{x_0}$).
\end{definizione}

\begin{oss}
We mention that another definition of quasiconvexity for maps defined on manifolds was given in~\cite{Li}.
\end{oss}

\begin{oss}
The main obstacle in generalizing to the Riemannian context the usual definition of quasiconvexity in the Euclidean ambient, is due to the fact that we cannot identify all the tangent spaces of a manifold as in $\R^n$, hence two differentials at two different points cannot be added together. Thus, in the definition above we needed to morally ``carry'' the differential of the perturbation at every point of the manifold to be an element of $T_{x_0}^*M$, via the differential of the exponential map at $x_0\in M$ (exponential map that in the Euclidean case would simply be the identity, under the identification mentioned above -- see also the following discussion). This forces the introduction of a ``correction term'' in the quasiconvexity inequality that one reasonably expects to go to zero as we get closer and closer to the point $x_0$, since the differential of the exponential map then tends to be the identity.
\end{oss}

If $(M,g)$ is $\R^n$ with its standard metric, for every point $x\in\R^n$ we have a standard identification $T^*_x\R^n\approx T_x{\R^n}\approx\R^n$, hence a function $f:\mathscr{L}(T\R^n, \mathbb{R}^m)\to \mathbb{R}$ can be one--to--one associated with a function $\widetilde{f}:\R^n\times\R^{n\times m}\to\R$ as follows: if $A=A_i^j\in \R^{n\times m}$, for $i\in\{1,\dots,n\}$ and $j\in\{1,\dots,m\}$, then, 
\begin{equation}\label{eq560}
\widetilde{f}({x_0},A)=f(\alpha_{x_0}),
\end{equation}
where $\alpha_{x_0}\in (T^*_{x_0}\R^n)^m\approx\R^{n\times m}$ is given by 
$$
\alpha^j_{x_0}(v)=\sum_{i=1}^n A_i^j v^i,
$$
for every vector $v=(v^1,\dots,v^n)\in T_{x_0}\R^n\approx\R^n$ and viceversa, if $\alpha_{x_0}$ is operating as in this formula, then $f$ is defined by equality~\eqref{eq560}.\\
Then, with this one--to--one correspondence, our ``Riemannian'' definition of quasiconvexity in the case of the Euclidean space, is equivalent to 
\begin{equation}\label{eq562}
\widetilde{f}\big(x_0,A\big)\leqslant\fint_{Q_{x_0}^r} \widetilde{f}\big(x,A+D\varphi(x)\big)\,d\L^n(x)+o(1),
\end{equation}
for every $x_0\in\R^n$, $A\in\R^{n\times m}$ and $\varphi\in C^\infty_c(Q_{x_0}^r,\mathbb{R}^m)$, since, under the identification $T^*_x\R^n\approx T_x{\R^n}\approx\R^n$, all the exponential maps are the identity and it is clearly satisfied by a continuous quasiconvex function $\widetilde{f}:\R^n\times\R^{n\times m}\to\R$, actually without the ``correction term'' $o(1)$, according to the usual Definition~\ref{quasidef} of quasiconvexity.\\
Hence, a ``classical'' quasiconvex function $\widetilde{f}$ in the Euclidean space, having as arguments only $x$ and $A$, is also quasiconvex in the sense of Definition~\ref{RiemannianQC} (with the above identification of $\widetilde{f}$ with $f$). Viceversa, if the function $f$ satisfies Definition~\ref{RiemannianQC}, then, by Theorem~\ref{QCR->sci} that we are going to show in the next section and the identification of $\widetilde{f}$ with $f$, it follows that the functional 
$$
u\mapsto F(u,\Omega)=\int_\Omega \widetilde{f}\big(x,Du(x)\big)\,d\L^n(x)
$$
is sequentially lower semicontinuous in the weak* topology of $W^{1,\infty}(\Omega,\R^m)$, hence $\widetilde{f}$ is quasiconvex according to the Definition~\ref{quasidef}, by the Acerbi--Fusco Theorem~\ref{Fusco1}.\\
Thus, our Definition~\ref{RiemannianQC} is an extension to Riemannian manifolds of the usual definition of quasiconvexity for continuous integrands.

\section{Quasiconvexity and semicontinuity}

We are going to show the analogues of some special cases of Theorem~\ref{Fusco1} of Acerbi and Fusco in our context, following the line of~\cite{Ambrosio} (we mention that the main argument in the next Theorem~\ref{QCR->sci} is due to Fonseca--Muller~\cite{FonsecaMuller}). We underline (see the discussion above) that we are generalizing the special case of integrands which are continuous and depend only on $x$ and $Du$ ({\em but not on $u$} -- see the final section), from the Euclidean to the Riemannian setting.

\medskip

The weak$^*$ convergence of a sequence $u_j\weakkk u$ in $W^{1,\infty}(\Omega,\R^m)$ is defined in the usual way: the sequences of integrals of $u_j$ and $du_j$ ``against'' fixed functions and $1$--forms in $L^1(\Omega,\R^m)$, respectively, converge to the analogous integrals relative to $u$ and $du$. Moreover, we notice that by the weak$^*$ sequential compactness of the closed unit ball of $L^{\infty}(\Omega,\R^m)$ given by the Banach--Alaoglu--Bourbaki theorem, we can simply ask that the sequence of differentials $du_j$ is bounded and that $u_j$ weakly$^*$ converges to $u$ in $L^{\infty}(\Omega,\R^m)$.

\begin{Teorema}\label{QCR->sci}
Let $(M,g)$ be a smooth, connected and complete Riemannian manifold and $\mu$ its canonical volume measure. Let $f:\mathscr{L}(TM, \mathbb{R}^m)\to \mathbb{R}$ be continuous and quasiconvex in the sense of Definition~\ref{RiemannianQC}. Then, for every open and bounded subset $\Omega\subseteq M$, the functional 
$$
u\mapsto F(u,\Omega)=\int_\Omega f(du)\,d\mu
$$
is sequentially lower semicontinuous in the weak* topology of $W^{1,\infty}(\Omega,\R^m)$, that is,
$$
F(u,\Omega)=\int_{\Omega}f\big(du\big)\,d\mu\leqslant\liminf_{j\to\infty }\int_{\Omega}f\big(du_j\big)\,d\mu=\liminf_{j\to\infty}F(u_j, \Omega),
$$
for every sequence $u_j\weakkk u$ in $W^{1,\infty}(\Omega,\R^m)$.
\end{Teorema}
\begin{proof}
Let $n\in\NN$ be the dimension of $M$. Let $x_0\in\Omega$ and we assume that $r>0$ is small enough for $Q_{x_0}^r$ to be contained in a neighbourhood $U\subseteq\Omega$ of $x_0$ such that the vector bundle $W^{1,\infty}(\Omega,\R^m)$ can be ``trivialized'' as we discussed above at the beginning of this section, by choosing an orthonormal frame $E_1,\dots,E_n$ in $U$, with an associated map $I$ and distance $\delta$ as in formulas~\eqref{mapI} and~\eqref{distL}. Moreover, since in what follows all the arguments of the continuous function $f$ will be bounded (since we are working in $W^{1,\infty}(\Omega,\R^m)$ and $u_j$ is a bounded sequence in $W^{1,\infty}$, being weakly$^*$ convergent), we can assume that $f$ is bounded and has a uniform modulus of continuity $\omega$ in $\mathscr{L}(TU, \mathbb{R}^m)$ (with respect to the distance $\delta$) which is continuous, bounded and concave, hence subadditive. Moreover, since it is bounded, by possibly adding a constant to $f$, we can also assume that it is positive.\\
We consider a smooth function $\psi: Q_{x_0}^r\to[0,1]$ with compact support and we set, for every $j\in\NN$, 
$$
\varphi_j=\psi(u_j-u)\in W_0^{1,\infty}(Q_{x_0}^r,\mathbb{R}^m),
$$
hence,
\begin{equation}\label{eq590}
d\varphi_j=\psi\,d (u_j-u)+d\psi\otimes (u_j-u).
\end{equation}
To simplify the notation in the computations below, we define
\begin{equation*}
L_x=d\exp_{x_0}[\exp_{x_0}^{-1}(x)],
\end{equation*}
for every $x\in Q_{x_0}^r$ and we start by applying the hypothesis of quasiconvexity for $f$, observing that inequality~\eqref{quasidefR} also holds for maps in $W_0^{1,\infty}(Q_{x_0}^r,\mathbb{R}^m)$, by approximating them with a sequence of functions in $C^\infty_c(Q_{x_0}^r,\mathbb{R}^m)$. Indeed, the function $o(1)$ depends, by hypothesis, in a monotonically nondecreasing way, on the $L^\infty$ norm of the gradient of such maps and this approximation can be chosen in such a way to control the $L^\infty$ norms of the gradients of the approximating functions with the $L^\infty$ norm of the gradient of the target function.\\
Hence, setting $\alpha_{x_0}=du[x_0]$ in inequality~\eqref{quasidefR}, for every $j\in\NN$, we have
\begin{align*}
f\big(du[x_0]\big)\leqslant&\,\fint_{Q_{x_0}^r}\!\!\!\!f\big(du[x_0]+d \varphi_j[x]\circ\!L_x\big)\,d\mu(x)+o_j(1)\\
=&\,\fint_{Q_{x_0}^r}\!\!\!\!f\big(du[x_0]+\psi\,d (u_j-u)[x]\circ\!L_x+d\psi[x]\otimes (u_j-u)\circ\!L_x\big)\,d\mu(x)+o_j(1)\\
=&\,\fint_{Q_{x_0}^r}\!\!\!\!f\big(du[x_0]-\psi\,du[x]\circ\!L_x\!+\psi\,du_j[x]\circ\!L_x\!+d\psi[x]\otimes (u_j-u)\circ\!L_x\big)\,d\mu(x)+o_j(1),
\end{align*}
where the function $o_j(1)$ depends monotonically on $\Vert d\varphi_j\Vert_\infty$, as in Definition~\ref{RiemannianQC}.\\
Then, by the properties of the modulus of continuity $\omega$ of $f$, there holds 
\begin{align*}
f\big(du[x_0]-&\,\psi\,du[x]\circ\!L_x+\psi\,du_j[x]\circ\!L_x+d\psi[x]\otimes (u_j-u)\circ\!L_x\big) 
-f\big(du_j[x]\circ\!L_x\big)\\
\leqslant&\,\omega\big(\big|du[x_0]-\psi\,du[x]\circ\!L_x+(\psi-1)du_j[x]\circ\!L_x+d\psi[x]\otimes (u_j-u)\circ\!L_x\big|_{g_{x_0}}\big)\\
\leqslant&\,\omega\big(\big|du[x_0]-\psi\,du[x]\circ\!L_x|_{g_{x_0}}\big)+\omega\big(\big|(\psi-1)du_j[x]\circ\!L_x\big|_{g_{x_0}}\big)\\
&\,+\omega\big(\big|d\psi[x]\otimes (u_j-u)\circ\!L_x\big|_{g_{x_0}}\big)\\
\leqslant&\,\omega\big(\big|du[x_0]-\psi\,du[x]\circ\!L_x|_{g_{x_0}}\big)+\omega\big(C|\psi-1|\big)+\omega\big(\big|d\psi[x]\otimes (u_j-u)\circ\!L_x\big|_{g_{x_0}}\big),
\end{align*}
where we kept into account that the distance $\delta$ on $\mathscr{L}(TU, \mathbb{R}^m)$, restricted to the fibers coincides with the one induced by the metric $g$ of $M$. Thus, we obtain 
\begin{align}
f\big(du[x_0]\big)\leqslant&\,\fint_{Q_{x_0}^r}f\big(du_j[x]\circ\!L_x\big)\,d\mu(x)+\fint_{Q_{x_0}^r}\omega\big(\big|du[x_0]-\psi\,du[x]\circ\!L_x|_{g_{x_0}}\big)\,d\mu(x)\\
&\,+\fint_{Q_{x_0}^r}\omega\big(C|\psi-1|\big)\,d\mu(x)+\fint_{Q_{x_0}^r}\omega\big(\big|d\psi[x]\otimes (u_j-u)\circ\!L_x\big|_{g_{x_0}}\big)\,d\mu(x)+o_j(1)\,.\\
&\,\ \label{eq556}
\end{align}
We then deal with the term 
$$
\fint_{Q_{x_0}^r}f\big(du_j[x]\circ\!L_x\big)\,d\mu(x)\leqslant\fint_{Q_{x_0}^r}f\big(du_j[x]\big)\,d\mu(x)+\fint_{Q_{x_0}^r}\omega\big(\delta\big(du_j[x]\circ\!L_x,du_j[x]\big)\big)\,d\mu(x),
$$
showing that the last integral is bounded by a function $o(1)$ independent of $j\in \NN$.\\
Indeed, recalling formula~\eqref{distL}, we have
\begin{align}
\omega\big(\delta\big(du_j[x]\circ\!L_x,du_j[x]\big)\big)
=&\,\omega\big(d^M(x_0,x\big)+\big\Vert\, Idu_j[x]\circ\!L_x - Idu_j[x]\,\big\Vert\big)\\
\leqslant&\,\omega\big(d^M(x_0,x\big)\big)+\omega\big(\big\Vert\, Idu_j[x]\circ\!L_x - Idu_j[x]\,\big\Vert\big)\\
\leqslant&\,\omega(r)+\omega\big(\big\Vert\, Idu_j[x]\circ\!L_x - Idu_j[x]\,\big\Vert\big)
\end{align}
and setting $\alpha_k=\big(Idu_j[x]\big)_k=du_j[x](E_k)\in\R^m$, for every $k\in\{1,\dots,n\}$, we have
\begin{align}
\big\Vert Idu_j[x]\circ\!L_x - Idu_j[x]\big\Vert=&\,\big\Vert I\big(du_j[x]\circ d\exp_{x_0}[\exp_{x_0}^{-1}(x)]\big)-Idu_j[x]\big\Vert\\
=&\,\bigg\Vert\sum_{k=1}^n\alpha_k \big(d\exp_{x_0}[\exp_{x_0}^{-1}(x)](E_i)\big)^k-\alpha_i\bigg\Vert\\
=&\,\bigg\Vert\sum_{k=1}^n\alpha_k\Big\{\big(d\exp_{x_0}[\exp_{x_0}^{-1}(x)](E_i)\big)^k-\delta_i^k\Big\}\bigg\Vert\\
\leqslant&\,C\Vert\alpha\Vert\,\big\Vert d\exp_{x_0}[\exp_{x_0}^{-1}(x)]-\mathrm{Id}_{T_{x_0}M}\big\Vert\\
=&\,C\big\Vert Idu_j[x]\big\Vert\,\big\Vert d\exp_{x_0}[\exp_{x_0}^{-1}(x)]-\mathrm{Id}_{T_{x_0}M}\big\Vert\\
=&\,\Vert du_j\Vert_{\infty}\,o(1)=o(1)\label{eq570}
\end{align}
for some constant $C$, as the norms $\Vert du_j\Vert_{\infty}$ are uniformly bounded.\\
Hence, inequality~\eqref{eq556} becomes
\begin{align}
f\big(du[x_0]\big)\leqslant&\,\fint_{Q_{x_0}^r}f\big(du_j[x]\big)\,d\mu(x)+\fint_{Q_{x_0}^r}\omega\big(\big|du[x_0]-\psi\,du[x]\circ\!L_x|_{g_{x_0}}\big)\,d\mu(x)+o(1)\\
&\,+\fint_{Q_{x_0}^r}\omega\big(C|\psi-1|\big)\,d\mu(x)+\fint_{Q_{x_0}^r}\omega\big(\big|d\psi[x]\otimes (u_j-u)\circ\!L_x\big|_{g_{x_0}}\big)\,d\mu(x)+o_j(1)\\
&\,\ \label{eq557}
\end{align}
where the function $o(1)$ is independent of $j\in\NN$, while the functions $o_j(1)$ are equal to $\eta(r,\Vert d\varphi_j\Vert)$ for some function $\eta$ going to zero as $r\to0$ and monotone nondecreasing (and we can also clearly assume continuous from the right) in its second argument.\\
Observing now that, by equation~\eqref{eq590}, there holds 
\begin{equation}
\Vert d\varphi_j\Vert_\infty\leqslant\Vert\psi\Vert_\infty\,\Vert du_j-du\Vert_\infty+\Vert d\psi\Vert_\infty\,\Vert u_j- u\Vert_\infty,
\end{equation}
we have
\begin{align}
o_j(1)=\eta(r,\Vert d\varphi_j\Vert_\infty)&\,\leqslant \eta(r,\Vert\psi\Vert_\infty\,\Vert du_j-du\Vert_\infty+\Vert d\psi\Vert_\infty\,\Vert u_j- u\Vert_\infty)\\
&\,\leqslant \eta(r,C+\Vert d\psi\Vert_\infty\,\Vert u_j- u\Vert_\infty),\label{eq595}
\end{align}
for some constant $C$ uniformly bounding from above $\Vert\psi\Vert_\infty\,\Vert du_j-du\Vert_\infty$. It follows that
$$
\limsup_{j\to\infty} \,o_j(1)\leqslant \limsup_{j\to\infty}\eta(r,C+\Vert d\psi\Vert_\infty\,\Vert u_j- u\Vert_\infty)=\eta(r,C)=o(1),
$$
with a function $o(1)$ independent of $j\in\NN$, by the properties of the function $\eta$ and the fact that $\Vert u_j- u\Vert_\infty\to0$ (by the theorem of Ascoli--Arzel\`a, being all the functions $u_j$ equibounded and equicontinuous). Then, passing to the liminf as $j\to\infty$ in formula~\eqref{eq557}, we obtain
\begin{align*}
f\big(du[x_0]\big)\leqslant&\,\liminf_{j\to\infty}\fint_{Q_{x_0}^r}f\big(du_j[x]\big)\,d\mu(x)+
\fint_{Q_{x_0}^r}\omega\big(\big|du[x_0]-\psi\,du[x]\circ\!L_x|_{g_{x_0}}\big)\,d\mu(x)\\
&\,+\fint_{Q_{x_0}^r}\omega\big(C|\psi-1|\big)\,d\mu(x)+o(1)\,,
\end{align*}
as the last integral in such formula goes to zero, by the dominated convergence theorem. Now letting $\psi$ go to the characteristic function of ${Q}_{x_0}^r$, the last integral in this formula vanishes (again by the dominated convergence theorem), hence
$$
f\big(du[x_0]\big)\leqslant\liminf_{j\to\infty}\fint_{Q_{x_0}^r}f\big(du_j[x]\big)\,d\mu(x)+
\fint_{Q_{x_0}^r}\omega\big(\big|du[x_0]-du[x]\circ\!L_x|_{g_{x_0}}\big)\,d\mu(x)+o(1)\,.
$$
Finally, we want to show that the second integral on the right is bounded by a function $o(1)$. By arguing as above, when we dealt with $du_j$, we have
\begin{align}
\omega\big(\big|du[x_0]-du[x]\circ\!L_x|_{g_{x_0}}\big)\leqslant&\,\omega\big(\delta\big(du[x_0],du[x]\big)+\delta\big(du[x],du[x]\circ\!L_x|_{g_{x_0}}\big)\big)\\
\leqslant&\,\omega\big(\delta\big(du[x_0],du[x]\big)\big)+\omega\big(\delta\big(du[x],du[x]\circ\!L_x|_{g_{x_0}}\big)\big)\\
=&\,\omega\big(d^M(x_0,x)\big)+\omega\big(\big\Vert Idu[x_0]-Idu[x]\big\Vert\big)+o(1)\\
=&\,\omega\big(\big\Vert Idu[x_0]-Idu[x]\big\Vert\big)+o(1),
\end{align}
hence,
$$
\fint_{Q_{x_0}^r}\omega\big(\big|du[x_0]-du[x]\circ\!L_x|_{g_{x_0}}\big)\,d\mu(x)\leqslant \fint_{Q_{x_0}^r}\omega\big(\big\Vert Idu[x_0]-Idu[x]\big\Vert\big)\,d\mu(x)+o(1).
$$
If now $x_0$ is a Lebesgue point for the function $Idu$, it is easy to show that the integral on the right goes to zero as $r\to 0$, being the function $t\mapsto \omega(t)$ continuous and bounded (and going to zero as $t\to0$). Thus, we conclude that at $\mu$--almost every point $x_0$ of $\Omega$, there holds 
\begin{equation}\label{eq558}
f\big(du[x_0]\big)\leqslant\liminf_{j\to\infty}\fint_{Q_{x_0}^r}f\big(du_j\big)\,d\mu+o(1)\,.
\end{equation}
As $f$ and $\Omega$ are bounded, the following integral is finite,
\begin{equation*}
\liminf_{j\to\infty} \int_{\Omega}f\big(du_j)d\mu=m<+\infty
\end{equation*}
and we can define the finite Radon measures $\nu_j$ on $\Omega$, for every $j\in\NN$, given by 
\begin{equation*}
d\nu_j=f\big(du_j\big)\,d\mu,
\end{equation*}
satisfying the uniform bound $\Vert\nu_j\Vert=\int_\Omega f\big(du_j\big)\,d\mu\leqslant C$. Hence, by the Banach--Alaoglu--Bourbaki theorem and without loss of generality, we may assume that $\nu_j$ weakly$^*$ converges to some limit Radon measure $\nu$.\\
For every open subset $G$ of $\Omega$, by the properties of the weak$^*$ convergence, we have
\begin{equation*}
\nu(G)\leqslant\liminf_{j\to\infty}\nu_j(G)=\liminf_{j\to\infty}\int_{G}f\big(du_j\big)\,d\mu\leqslant\liminf_{j\to\infty} \int_{\Omega}f\big(du_j\big)\,d\mu=m,
\end{equation*}
hence, by the outer regularity of the measure $\mu$, it follows that $\nu\ll\mu$. Thus, we may apply the Radon--Nikodym theorem, obtaining a function $h:\Omega\to [0,+\infty]$ such that $d\nu=h\,d\mu$ and
\begin{equation*}
h(x)=\lim_{r\to 0} \frac{\nu\big(\overline{Q}_{x}^r\big)}{\mu\big(\overline{Q}_{x}^r\big)},
\end{equation*}
for $\mu$--almost every $x\in\Omega$.\\
Then, at the points $x_0\in\Omega$ where inequality~\eqref{eq558} is satisfied and the above limit holds, we have
\begin{align*}
f\big(du[x_0]\big)\leqslant&\,\liminf_{r\to0}\bigg(\liminf_{j\to\infty}\fint_{Q_{x_0}^r}f\big(du_j\big)\,d\mu+o(1)\bigg)=\liminf_{r\to0}\liminf_{j\to\infty}\frac{\nu_j\big(Q^r_{x_0}\big)}{\mu\big(Q^r_{x_0}\big)}\\
=&\,\liminf_{r\to0}\liminf_{j\to\infty}\frac{\nu_j\big(\overline{Q}^r_{x_0}\big)}{\mu\big(\overline{Q}^r_{x_0}\big)}\leqslant\liminf_{r\to0}\frac{\nu\big(\overline{Q}^r_{x_0}\big)}{\mu\big(\overline{Q}^r_{x_0}\big)}=h(x_0),
\end{align*}
as $\liminf_{j\to\infty}\nu_j\big(\overline{Q}^r_{x_0}\big)\leqslant \nu\big(\overline{Q}^r_{x_0}\big)$, being $\overline{Q}^r_{x_0}$ closed sets. Hence, $f\big(du[x]\big)\leqslant h(x)$ $\mu$--almost everywhere in $\Omega$, thus
$$
F(u,\Omega)\!=\!\fint_\Omega f\big(du\big)\,d\mu\!\leqslant\!\fint_\Omega h\,d\mu\!=\!\nu(\Omega)\!\leqslant\!\liminf_{j\to\infty}\nu_j(\Omega)\!=\!\liminf_{j\to\infty}\fint_\Omega f\big(du_j\big)\,d\mu\!=\!\liminf_{j\to\infty}F(u_j,\Omega)
$$
and the proof is complete.
\end{proof}

The following theorem is the converse of the previous one.

\begin{Teorema}\label{sci->QCR}
Let $(M,g)$ be a smooth, connected and complete Riemannian manifold and $\mu$ its canonical volume measure. Let $f:\mathscr{L}(TM, \mathbb{R}^m)\to \mathbb{R}$ be a continuous function. If for every open and bounded subset $\Omega\subseteq M$, the functional
\begin{equation*}
F(u,\Omega)=\int_{\Omega}f\big(du\big)\,d\mu
\end{equation*}
is sequentially lower semicontinuous in the weak* topology of $W^{1,\infty}(\Omega,\R^m)$, then the function $f$ is quasiconvex in the sense of Definition~\ref{RiemannianQC}. 
\end{Teorema}
\begin{proof}
Let $n\in\NN$ be the dimension of $M$. We adopt the same setting and notation as in the proof of Theorem~\ref{QCR->sci}, in particular, fixed $x_0\in M$ and $\alpha_{x_0}\in(T_{x_0}^*M)^m$, we work in a neighbourhood $U\subseteq M$ of $x_0$ such that the vector bundle $\mathscr{L}(TU, \mathbb{R}^m)$ can be ``trivialized'' and $Q_{x_0}^r\subseteq U$.\\
We consider a smooth function $u:M\to\R^m$ such that $du[x_0]=\alpha_{x_0}$ (which clearly exists) and let $\varphi\in C^{\infty}_c\big(Q_{x_0}^r,\R^m\big)$, then the function $\varphi\circ\exp_{x_0}$ is defined on the open cube $r\left(-\frac12,\frac12\right)^n\subseteq T_{x_0}M$, smooth and with compact support.
We define the extension by periodicity $\psi:T_{x_0}M\to\R^m$ of $\varphi\circ\exp_{x_0}$ to the whole $T_{x_0}M$ (defined zero on the boundaries of the cubes) and we set
$$
\varphi_h(x)=\frac{1}{h}\psi(h\exp_{x_0}^{-1}(x))
$$
for every $x\in Q_{x_0}^r$ and $h\in\NN$.
Then, the functions $\varphi_h:Q_{x_0}^r\to\R^m$ are smooth with compact support and 
\begin{equation*}
\varphi_h\weakkk0\quad \text{ in }W^{1,\infty}_0(Q_{x_0}^r, \mathbb{R}^m)
\end{equation*}
as $h\to\infty$, indeed, clearly $\varphi_h\to 0$ and the differentials $d \varphi_h$ are uniformly bounded, holding
$$
d\varphi_h[x]=\frac{1}{h}d\psi\big[h\exp_{x_0}^{-1}(x)\big]h \circ d\exp_{x_0}^{-1}[x]=d\psi\big[h\exp_{x_0}^{-1}(x)\big]\circ d\exp_{x_0}^{-1}[x].
$$
Thus, since $u+\varphi_h\weakkk u$ in $W^{1,\infty}(Q_{x_0}^r, \mathbb{R}^m)$, by the hypothesis of lower semicontinuity of the functional $F$, there holds
\begin{align*}
\int_{Q_{x_0}^r}f\big(du[x]\big)\,d\mu(x)&\leqslant \liminf_{h\to\infty} \int_{Q_{x_0}^r}f\big(du[x]+d \varphi_h[x]\big)\,d\mu(x)\\
&=\liminf_{h\to\infty}\int_{Q_{x_0}^r}f\big(du[x]+d\psi\big[h\exp_{x_0}^{-1}(x)\big]\circ d\exp_{x_0}^{-1}[x]\big)\,d\mu(x),
\end{align*}
for every $r>0$ small enough. Then, arguing as in the proof of Theorem~\ref{QCR->sci} (more precisely, following the same argument leading to estimate~\eqref{eq570}), considering a suitable modulus of continuity $\omega$ of $f$, we have
\begin{align}
\big|f\big(du&\,[x]+d\psi\big[h\exp_{x_0}^{-1}(x)\big]\circ d\exp_{x_0}^{-1}[x])-f\big(du[x_0]+d\psi\big[h\exp_{x_0}^{-1}(x)\big]\circ \mathrm{Id}_{T_{x_0}M}\big)\big|\\
&\,\leqslant\omega\big(\delta\big(du[x],du[x_0]\big)\big)+
\omega\big(\delta\big(d\psi\big[h\exp_{x_0}^{-1}(x)\big]\circ d\exp_{x_0}^{-1}[x],d\psi\big[h\exp_{x_0}^{-1}(x)\big]\big)\big)\\
&\,=\omega\big(d^M(x,x_0)\big)+\omega\big(\big\Vert Idu[x]-Idu[x_0]\big\Vert\big)+o(1)\Vert d\psi\Vert_\infty\\
&\,=o(1)(1+\Vert d\psi\Vert_\infty)\label{eq580}
\end{align}
as $r\to 0$, since the function $x\mapsto Idu[x]$ is smooth. Hence,
\begin{equation}\label{eq572}
\int_{Q_{x_0}^r}f\big(du[x]\big)\,d\mu(x)\leqslant\liminf_{h\to\infty}\int_{Q_{x_0}^r}f\big(du[x_0]+d\psi\big[h\exp_{x_0}^{-1}(x)\big]\big)\,d\mu(x)+o(1)\mu(Q_{x_0}^r).
\end{equation}
Changing variables as $y=\exp_{x_0}^{-1}(x)$, so $y\in r\big(-\frac{1}{2}, \frac{1}{2} \big)^n\subseteq T_{x_0}M$ and denoting by $dy$ the Lebesgue measure on $T_{x_0}M$ relative to the metric tensor $g_{x_0}$, we have
\begin{align}
\int_{Q_{x_0}^r}f\big(du[x_0]+d\psi\big[h\exp_{x_0}^{-1}(x)\big]\big)\,d\mu(x)=&\,\int_{r\big(-\frac{1}{2}, \frac{1}{2}\big)^n}f\big(du[x_0]+d\psi[hy]\big)\,\mathrm{J}(y)\,dy\\
\leqslant&\,\int_{r\big(-\frac{1}{2}, \frac{1}{2}\big)^n}f\big(du[x_0]+d\psi[hy]\big)\,dy+Cr^{n+1},\\
&\,\label{eq573}
\end{align}
for some constant $C$ independent of $h\in\NN$, since the Jacobian 
$$
\mathrm{J}(y)=\Big|\det_{ij}\big(g_{\exp_{x_0} \!(y)}(d\exp_{x_0}[y](E_i),E_j)\big)\Big|
$$
is a smooth function and goes uniformly to $1$ as $r\to 0$, being $|y|\leqslant r\sqrt{n}/2$ (we recall that $d\exp_{x_0}[0])$ is the identity of $T_{x_0}M$) and $f\big(du[x_0]+d\psi[hy]\big)$ is bounded. Then, changing the variables again as $w=hy$ in the last integral, we get
$$
\int_{Q_{x_0}^r}f\big(du[x_0]+d\psi\big[h\exp_{x_0}^{-1}(x)\big]\big)\,d\mu(x)\leqslant\frac{1}{h^n}\int_{hr\big(-\frac{1}{2}, \frac{1}{2}\big)^n} f\big(du[x_0]+d\psi [w]\big)\,dw+Cr^{n+1}.
$$
Now, by the periodicity of $\psi$, there holds
$$
\frac{1}{h^n}\int_{hr\big[-\frac{1}{2}, \frac{1}{2}\big]^n} f\big(du[x_0]+d\psi [w]\big)\,dw=\frac{1}{h^n}h^n\int_{r\big[-\frac{1}{2}, \frac{1}{2}\big]^n} f\big(du[x_0]+d \varphi [\exp_{x_0}(w)]\circ d\exp_{x_0}[w] \big)\,dw
$$
and if we change (back) variables as $x=\exp_{x_0}(w)$, we obtain
\begin{align*}
\int_{r\big[-\frac{1}{2}, \frac{1}{2}\big]^n}&\, f\big(du[x_0]+d \varphi [\exp_{x_0}(w)]\circ d\exp_{x_0}[w] \big)\,dw\\
=&\,\int_{Q_{x_0}^r} f\big(du[x_0]+d \varphi [x]\circ d\exp_{x_0}[\exp_{x_0}^{-1}(x)])\mathrm{J}(x)^{-1}\,d\mu(x)\\
\leqslant&\,\int_{Q_{x_0}^r} f\big(du[x_0]+d \varphi [x]\circ d\exp_{x_0}[\exp_{x_0}^{-1}(x)]\big)\,d\mu(x)+Cr\mu(Q_{x_0}^r),
\end{align*}
arguing as above.
Hence, by this inequality where $h\in\NN$ is not present in the right hand side and formulas~\eqref{eq572},~\eqref{eq573}, we conclude
$$
\int_{Q_{x_0}^r}f\big(du[x]\big)\,d\mu(x)\leqslant\int_{Q_{x_0}^r} f\big(du[x_0]+d \varphi [x]\circ d\exp_{x_0}[\exp_{x_0}^{-1}(x)]\big)\,d\mu(x)+o(1)\mu(Q_{x_0}^r),
$$
as $\mu(Q_{x_0}^r)\approx \omega_nr^n$, where $\omega_n$ is the measure of the unit ball of $\R^n$.\\
Since, by the continuity of $du$, it easily follows that
\begin{equation*}
\int_{Q_{x_0}^r}f\big(du[x_0]\big)\,d\mu(x)\leqslant\int_{Q_{x_0}^r}f\big(du[x]\big)\,d\mu (x)+o(1)\mu(Q_{x_0}^r)
\end{equation*}
and $du[x_0]=\alpha_{x_0}$, we finally have
\begin{align*}
f(\alpha_{x_0})&\leqslant\frac{1}{\mu(Q_{x_0}^r)}\int_{Q_{x_0}^r} f\big(\alpha_{x_0}
+d \varphi [x]\circ d\exp_{x_0}[\exp_{x_0}^{-1}(x)]\big)\,d\mu(x)+o(1)
\end{align*}
for every $x_0\in M$ and with $o(1)$ going to zero as $r\to 0$, depending only on the $L^{\infty}$ norm of $d\varphi$. Moreover, by tracing back how we obtained such function $o(1)$, it is clear that it can be chosen in a way that such dependence is monotonically nondecreasing (see in particular, estimate~\eqref{eq580}). Hence, the function $f$ is 
quasiconvex according to Definition~\ref{RiemannianQC}.
\end{proof}

Putting together the two theorems, we have the following one which generalizes (in our case of $f$ continuous) the results of Acerbi and Fusco for $p=+\infty$. 

\begin{Teorema}\label{QCR=sci}
Let $(M,g)$ be a smooth, connected and complete Riemannian manifold and $\mu$ its canonical volume measure. Let $f:\mathscr{L}(TM, \mathbb{R}^m)\to \mathbb{R}$ be continuous. 
Then, $f$ is quasiconvex in the sense of Definition~\ref{RiemannianQC} if and only if for every open and bounded subset $\Omega\subseteq M$, the functional 
$$
u\mapsto F(u,\Omega)=\int_\Omega f(du)\,d\mu
$$
is sequentially lower semicontinuous in the weak* topology of $W^{1,\infty}(\Omega,\R^m)$.
\end{Teorema}

\section{Some remarks and possible research directions}

A natural continuation of our work would be to extend the previous results to the $W^{1,p}(\Omega, \mathbb{R}^m)$ setting, with $p\in[1,+\infty)$, in order to obtain the analogue of Theorem~\ref{Teocap1pfinito} (notice that the fact that the semicontinuity of the functional implies quasiconvexity of $f$ follows immediately by Theorem~\ref{sci->QCR} in such setting, as in the Euclidean case).

Focardi and Spadaro in~\cite{FocardiSpadaro} showed an extension of the Acerbi--Fusco theorems (with $f$ continuous, like us) in the case of Sobolev maps $u:\Omega\to M$, with $\Omega$ an open bounded subset of $\mathbb{R}^n$ and $M$ a Riemannian manifold. Therefore, a possible future research line could be to ``combine'' Theorem~\ref{QCR=sci} with the results of Focardi and Spadaro, thus obtaining a characterization of quasiconvexity in the completely Riemannian context of Sobolev maps between two Riemannian manifolds.

Moreover, the full generalization of the Acerbi--Fusco results would be to show them for the ``naturally defined'' Carath\'eodory functions on a Riemannian manifold, that is, functions which are Carath\'eodory in the usual way after expressing them by means of the local ``trivializations'' of the bundle $\mathscr{L}(TM, \mathbb{R}^m)$ (or of $\mathscr{L}(TM,TN)$ after ``combining'' our definition with the one of Focardi--Spadaro), as we did at the beginning of this section.

Finally, it would be natural to introduce also the concepts of polyconvexity and rank--one convexity in the Riemannian context and compare them with our definition of quasiconvexity, reasonably obtaining the same relations holding in the Euclidean case.

\bibliographystyle{amsplain}
\bibliography{References.bib}

\end{document}